\documentclass[10pt]{amsart}
\usepackage{amsmath,amsfonts,amsthm,amssymb, hyperref, stmaryrd, mathrsfs}
\usepackage{enumitem,textcomp}
\theoremstyle{plain}
\newtheorem{theorem}{Theorem}[section]
\newtheorem{conjecture}[theorem]{Conjecture}
\newtheorem{lemma}[theorem]{Lemma}
\newtheorem{proposition}[theorem]{Proposition}
\newtheorem{corollary}[theorem]{Corollary}

\theoremstyle{remark}
\newtheorem{example}[theorem]{Example}
\newtheorem{remark}[theorem]{Remark}

\theoremstyle{definition}
\newtheorem{definition}[theorem]{Definition}
\newcommand{\Crit}{ \mathrm{Crit} }
\newcommand{\ft}{ \mathfrak{t} }
\newcommand{\ftd}{ \mathfrak{t}^\ast }

\newcommand{\fk}{ \mathfrak{k} }
\newcommand{\fkc}{ \mathfrak{k}_\mathbb{C} }
\newcommand{\fkcd}{ \mathfrak{k}_\mathbb{C}^\ast }
\newcommand{\fkd}{ \mathfrak{k}^\ast }

\newcommand{\npvz}{\norm{\normsq{\phi(v)}-\normsq{\phi_c}}}
\newcommand{\red}{/\!/}
\newcommand{\reda}[1]{ \underset{#1}{/\!/} }
\newcommand{\rred}{/\!/\!/}
\newcommand{\rreda}[1]{ \underset{#1}{/\!/\!/} }

\newcommand{\iso}{ \cong } 
\newcommand{\Span}{\mathrm{span}}

\newcommand{\mur}{\mu_\mathbb{R}}

\newcommand{\muc}{\mu_\mathbb{C}}

\newcommand{\muhk}{\mu_{\mathrm{HK}}}

\newcommand{\omegar}{ \omega_\mathbb{R} }
\newcommand{\omegac}{ \omega_\mathbb{C} }

\newcommand{\inner}[2]{\langle #1, #2 \rangle}

\newcommand{\grad}{\nabla}
\newcommand{\loj}{{\L}ojasiewicz}
\newcommand{\gli}{global {\L}ojasiewicz inequality}
\newcommand{\hk}{hyperk\"ahler}

\newcommand{\Ka}{K\"ahler}
\newcommand{\ab}{(\alpha, \beta)}
\newcommand{\CP}{ \mathbb{CP} }
\newcommand{\CC}{ \mathbb{C} }
\newcommand{\RR}{ \mathbb{R} }
\newcommand{\QQ}{ \mathbb{Q} }
\newcommand{\HH}{ \mathbb{H} }

\newcommand{\ZZ}{ \mathbb{Z} }

\newcommand{\norm}[1]{\left| #1 \right|}
\newcommand{\normsq}[1]{\left| #1 \right|^2}

\newcommand{\stab}{ \mathrm{stab} }

\newcommand{\suchthat}{ \ | \ }

\newcommand{\into}{ \hookrightarrow }
 
\newcommand{\Ad}{ \mathrm{Ad} }

\newcommand{\ann}{\mathrm{ann}} 
\begin{document}
%\pagestyle{headings}
%\begin{titlepage}  
\title[Morse Theory with the Norm-Square of a Hyperk\"ahler Moment Map]
{Morse Theory with the Norm-Square of a Hyperk\"ahler Moment Map}
\date{\today} 
%\end{titlepage}

\author{Jonathan Fisher}  
\address{Department of Mathematics, University of Toronto, Canada}
\email{jmfisher@math.utoronto.ca}

\begin{abstract} We prove that the norm-square of a moment
map associated to a linear action of a compact group on an affine variety
satisfies a certain gradient inequality. This allows us to bound the gradient
flow, even if we do not assume that the moment map is proper.
We describe how this inequality can be extended to hyperk\"ahler moment maps
in some cases, and use Morse theory with the norm-squares of 
hyperk\"ahler moment maps to compute the Betti numbers and cohomology rings
of all toric hyperk\"ahler orbifolds.
\end{abstract}

\maketitle

% TOC just helps navigate sections while editing   
%\setcounter{page}{1}
%\tableofcontents
%---------------------------------------------------------------------

%---------------------------------------------------------------------
\section{Introduction}
%---------------------------------------------------------------------
Let $M$ be a symplectic manifold acted upon by a compact group $K$ 
with moment map $\mu: M \to \fkd$. If $0$ is a regular value of $\mu$, 
then the reduced space $M \red K := \mu^{-1}(0)/G$ is a symplectic 
orbifold. Following the ideas of Atiyah and Bott \cite{YangMillsRiemannSurface}, 
Kirwan showed  that the 
topology of $M \red K$ can be understood via $K$-equivariant Morse theory
with $|\mu|^2$ \cite{KirwanCohomology}. Let $H_K^\ast(M)$ denote the
$K$-equivariant cohomology of $M$ (we take cohomology with rational 
coefficients unless otherwise specified).
There is a natural map $\kappa: H_K^\ast(M) \to H^\ast(M \red K)$
called the \emph{Kirwan map}, and one of the most important consequences 
of Kirwan's analysis is that (under mild assumptions) this map is surjective.
Thus we have an isomorphism
$H^\ast(M \red K) \iso H_K^\ast(M) / \ker \kappa$,
so that if $H_K^\ast(M)$ is known then the problem of computing the cohomology 
ring of $M \red K$ reduces to that of computing the kernel of $\kappa$,
 a problem which has been well-studied (see, for example, 
\cite{NonabelianLocalization, TolmanWeitsman}).

If $M$ is a hyperk\"ahler manifold
then a hyperk\"ahler moment map is a certain triple of moment maps
 $\muhk = (\mu_1, \mu_2, \mu_3)$, and the hyperk\"ahler quotient \cite{HKLR} is defined
to be $M \rred K := \muhk^{-1}(0) / K$ (we review this
construction in Section \ref{Quotients}). As in the symplectic case,
we have a natural map $H_K^\ast(M) \to H^\ast(M \rred K)$ which we call
the \emph{hyperk\"ahler Kirwan map}. 
It has been an outstanding question whether the hyperk\"ahler Kirwan map is 
surjective. There are several families of examples for which surjectivity is
known (see for example \cite{KonnoHypertoric, KonnoHyperpolygon}); 
however there is currently no general result.
Recently, it was proved \cite{CohomologyofHKQuotients} that surjectivity
of the hyperk\"ahler Kirwan map would follow from Morse theory with the 
function $|\muc|^2 := |\mu_2|^2 + |\mu_3|^2$, provided that its gradient flow 
satisfies 
certain analytic conditions. In principle, this provides a general method for 
proving surjectivity of the hyperk\"ahler Kirwan map.

However, the analytic
conditions needed to apply these methods appear
to be very difficult to check---in \cite{CohomologyofHKQuotients},
these conditions are verified only for linear $S^1$ actions,
but are conjectured to hold for all linear actions. The difficulty is that
unlike the symplectic case, we cannot assume that the moment map is proper,
as virtually all interesting hyperk\"ahler quotients are non-compact.
 
Motivated by this problem, our goal is to provide 
tools to study the gradient flow of $|\mu|^2$ without making any 
compactness assumptions. 
To control the gradient flow, we introduce a certain gradient inequality,
motivated by the works
 \cite{GradientSemialgebraic,LermanGradientFlow,NeemanQuotientVarieties}.
The main result in this direction is Theorem \ref{GlobalEstimate}, 
which gives a precise gradient inequality for $|\mu|^2$ when $\mu$ is a 
moment map associated to a linear action. We introduce the
relevant gradient inequalities in Section \ref{LojasiewiczInequalities}
and discuss some immediate applications.
We prove Theorem \ref{GlobalEstimate}
in Section \ref{Proof}. Combined with Proposition \ref{FlowClosedness},
this gives us rather good quantitative control on the gradient trajectories
of $-|\mu|^2$.

When the group $K$ is abelian, we also obtain analogous results for 
$|\muhk|^2$ and $|\muc|^2$. We use these results to show that the analytic 
conditions of 
\cite{CohomologyofHKQuotients} are satisfied by all linear actions by tori.
In Section \ref{HypertoricSection}, we study this case in detail
and show how Morse theory with the hyperk\"ahler moment map can be used
to compute the cohomology rings of all toric
hyperk\"ahler orbifolds, reproducing several known results
\cite{BielawskiDancerHypertoric,HaradaProudfoot,KonnoHypertoric} in a 
rather uniform way.

\textbf{Acknowledgements}. The author would like to thank
Lisa Jeffrey for her careful guidance and supervision, as well as
Andrew Dancer, Pierre Milman, and Nicholas Proudfoot for
helpful conversations.
This research was supported by an NSERC Canada Graduate Scholarship (Doctoral).
%---------------------------------------------------------------------
 
%---------------------------------------------------------------------
\section{K\"ahler and Hyperk\"ahler Quotients}
\label{Quotients}
%---------------------------------------------------------------------
Let $K$ be a compact Lie group acting on a symplectic manifold $(M, \omega)$.
Let $\fk$ denote the Lie algebra of $K$.
The action is said to be \emph{Hamiltonian} if there is a 
$K$-equivariant map $\mu: M \to \fkd$ which satisfies
\begin{equation}
d\inner{\mu}{\xi} = i(v_\xi) \omega \ \ \forall \xi \in \fk,
\end{equation}
where $i(\cdot)$ denotes the interior product and $v_\xi$ is the 
fundamental vector field
\begin{equation}
v_\xi(x) = \left. \frac{d}{dt} \right|_{t=0} e^{t\xi} \cdot x,
\end{equation}
associated to any $\xi \in \fk$. Since $K$ is compact, we can choose an 
invariant inner product on $\fk$ to identify it with $\fkd$, and via this 
identification we will often think of moment maps as taking values in $\fk$.

If $\alpha$ is a regular central value of $\mu$, then the
\emph{symplectic reduction}
\[ M \reda{\alpha} K := \mu^{-1}(\alpha) / K, \]
is a symplectic orbifold, and a manifold if $K$ acts freely on
$\mu^{-1}(0)$. We will sometimes denote this by $M \red K$ when we do not wish to
emphasize the parameter $\alpha$. If in addition $M$ admits a 
$K$-invariant \Ka\ metric $g$ compatible with $\omega$, 
the \Ka\ structure descends to $M \red K$ in a natural way.
In this case we call $M \red K$ a \emph{K\"ahler quotient} to
emphasize this fact.

A manifold $M$ is \hk\ if it has a metric $g$ and a triple of 
symplectic forms $(\omega_1, \omega_2, \omega_3)$ which are all K\"ahler
forms with 
respect to $g$ (i.e. compatible and parallel),
 and such that the respective complex structures 
$(I_1, I_2, I_3)$ satisfy the quaternion relations ($I_1 I_2 = I_3$, etc.). 
A hyperk\"ahler moment map for an action of $K$ on $M$ is a triple 
$\muhk := (\mu_1, \mu_2, \mu_3)$ such that each $\mu_i$ is a moment map
for the action of $K$ with respect to $\omega_i$. Let us introduce
\begin{eqnarray}
\omegar &=& \omega_1, \\
\omegac &=& \omega_2 + i \omega_3,
\end{eqnarray}
with corresponding moment maps $\mur := \mu_1$ and $\muc := \mu_2 + i \mu_3$.
We call $\mur$ and $\muc$ the real and complex moment maps, respectively.
We will often think of $\mu_{HK}$ as the pair $(\mur, \muc)$.
For a pair $(\alpha, \beta) \in \fkd \oplus \fkcd$ that is regular and
central, we define the \emph{hyperk\"ahler quotient} \cite{HKLR} to be
\[ M \rreda{\ab} K := ( \mur^{-1}(\alpha) \cap \muc^{-1}(\beta) ) / K. \]
This is a \hk\ orbifold, and a manifold if $K$ acts
freely on $\mur^{-1}(\alpha) \cap \muc^{-1}(\beta)$. 

We now recall the most important source of examples of hyperk\"ahler
quotients. Let $V$ be a Hermitian vector space. Its cotangent
bundle $T^\ast V$ is naturally a hyperk\"ahler vector space, with metric $g$ and
complex structure $I_1$ induced from $V$, with $I_2$ defined by
$I_2(x,y) = (-\bar{y}, \bar{x})$, and with $I_3 := I_1 I_2$.
Any unitary representation of $K$ on $V$ induces a linear action on
$T^\ast V$ which preserves this hyperk\"ahler structure.
Since $T^\ast V$ is a vector space, a moment map always exists, and
we may take the hyperk\"ahler quotient $T^\ast V \rred K$.
For a general discussion of hyperk\"ahler analogues of K\"ahler quotients, 
see \cite{ProudfootThesis}.

\begin{example} Let $V = \mathbb{C}^{N+1}$ with the standard diagonal
action of $S^1$. The moment map $\mu: V \to \mathbb{R}$ is given by
$x \mapsto |x|^2 / 2$, and the K\"ahler quotient is $\CP^N$ 
equipped with the Fubini-Study metric. For the induced action of $S^1$
on $T^\ast V$, we have
\begin{eqnarray*}
\mur &=& \frac{1}{2} |x|^2 - \frac{1}{2} |y|^2, \\
\muc &=& x \cdot y,
\end{eqnarray*}
and the hyperk\"ahler quotient $T^\ast V \rred S^1$ is $T^\ast \CP^N$
 with the Calabi metric. Note that while the moment map for the $S^1$ action
on $\CC^{N+1}$ is a square, the moment map for the induced action on 
$T^\ast \CC^{N+1}$ is a \emph{difference} of squares. 
\end{example}

\begin{remark} In the above example, we found that
\[ T^\ast \CC^{N+1} \rred S^1 \iso T^\ast (\CC^{N+1} \red S^1). \]
It is always true that $T^\ast (V \red K)$ embeds into $T^\ast V \rred K$
as an open dense set (as long as it is not empty), but in general
this is not an isomorphism.
\end{remark}
  
%---------------------------------------------------------------------
\section{Equivariant Morse Theory}
%---------------------------------------------------------------------
We now recall some of the main ideas of Kirwan \cite{KirwanCohomology}.
Let $M$ be a symplectic manifold with a Hamiltonian action of
a compact group $K$, and assume that $0$ is a regular value of the moment map.
Let $f = |\mu|^2$, and for the moment assume that $f$ is proper.

Kirwan shows that the components $C$ of the critical set of $f$ can
be understood in terms of the fixed-point sets of subtori of a maximal torus
of $K$. Using a local description of $|\mu|^2$ near each component $C$, she
then proves that $|\mu|^2$ is \emph{minimally degenerate} (a weaker condition
than being Morse-Bott). This allows her to conclude that the stable manifolds
$S_C$ (the set of points $x \in M$ such that 
the gradient flow of $-|\mu|^2$ through $x$ has a limit point in $C$)
form a smooth $K$-invariant stratification of $M$.
Thus for each $C$ we obtain the equivariant Thom-Gysin sequence
\[ \cdots \to H^{\ast-\lambda_C}_K(C) \to
 H^\ast_K(M_C^+) \to H^\ast_K(M_C^-) \to \cdots \]
where $\lambda_C$ is the Morse index of $C$ and 
$M_C^\pm = f^{-1}((-\infty, f(C)\pm\epsilon])$, for $\epsilon > 0$ sufficiently
small. The map $H^{\ast-\lambda_C}_K(C) \to H^\ast_K(M_C^+)$
is, after composition with restriction $H^\ast_K(M_C^+) \to H^\ast_K(C)$,
multiplication by the equivariant Euler class of the negative normal bundle
to $C$. By the Atiyah-Bott lemma \cite{MomentMapEquivariantCohomology}, 
the equivariant
Euler class is not a zero divisor, hence the equivariant Thom-Gysin sequence 
splits into short exact sequences
\[ 0 \to H^{\ast-\lambda_C}_K(C) \to
 H^\ast_K(M_C^+) \to H^\ast_K(M_C^-) \to 0. \]
Thus the restriction $H_K^\ast(M_C^+) \to H_K^\ast(M_C^-)$ is surjective,
and its kernel consists of those classes whose restriction to $C$ is 
a multiple of the equivariant Euler class of the negative normal bundle to $C$.
An inductive argument then yields surjectivity of the restriction
$H_K^\ast(M) \to H_K^\ast(\mu^{-1}(0)) \iso H^\ast(M \red K)$.

We are interested in the case when $|\mu|^2$ is not proper. 
To aid the discussion, we introduce the following definition.

\begin{definition} \label{DefnFlowClosed}
A function $f: M \to \RR$ is said to be \emph{flow-closed} if every
(positive time) trajectory of $-\grad f$ is contained in a compact set.
\end{definition}

As remarked
in \cite[\S 9]{KirwanCohomology}, the above results still hold under
the weaker condition that $|\mu|^2$ is flow-closed. We summarize
these results as follows.

\begin{theorem}[Kirwan \cite{KirwanCohomology}] \label{KirwanSurjectivity}
Suppose $f = |\mu|^2$ is flow-closed and that $0$ is a regular value of $\mu$. 
Then the stable manifolds $S_C$ form a smooth $K$-invariant stratification of 
$M$, and the equivariant Thom-Gysin sequence splits into short exact sequences.
Consequently, the function $|\mu|^2$ is equivariantly perfect,
and the Kirwan map $\kappa: H_K^\ast(M) \to H^\ast(M \red K)$ is surjective.
The kernel of $\kappa$ is the ideal in $H_K^\ast(M)$
generated by those classes whose restriction to some component $C$ of the 
critical set of $f$ is the equivariant 
Euler class of the negative normal bundle to $C$.
\end{theorem}

One might expect analogous results to hold for $|\muhk|^2$. However,
this appears not to be the case (except when $K$ is a 
torus \cite{KirwanHyperkahlerManuscript}). The problem is that in the
nonabelian situation, the norm of the gradient of $|\muhk|^2$ contains 
a term proportional to
the structure constants of the group that is difficult to understand
(see Remark \ref{CrossTermRemark}).
In \cite{CohomologyofHKQuotients}, it was found instead that the 
function $|\muc|^2$ is better behaved, owing to the fact that $\muc$
is $I_1$-holomorphic. 

\begin{theorem}[{Jeffrey-Kiem-Kirwan 
\cite{CohomologyofHKQuotients}, Kirwan \cite{KirwanHyperkahlerManuscript}}]
\label{SurjectivityCriterion} The function $f = |\muc|^2$ is minimally
degenerate. If $f$ is flow-closed, then the stable manifolds $S_C$
form a $K$-invariant stratification of $M$ and the equivariant Thom-Gysin
sequence splits into short exact sequences. Consequently, the restriction
$H_K^\ast(M) \to H_K^\ast(\muc^{-1}(0))$ is surjective.
If $K$ is a torus, the same 
conclusions hold for $|\muhk|^2$ provided that it is flow-closed.
\end{theorem}  

Note that $\mu^{-1}(0)$ is a $K$-invariant complex submanifold of $M$,
so that $M \rred K \iso \muc^{-1}(0) \red K$. Hence surjectivity of map
$H_K^\ast(\muc^{-1}(0)) \to H^\ast(M \rred K)$ can be studied by
the usual methods (but see Corollary \ref{HomotopyCorollary}).
In principle, this reduces the question of Kirwan surjectivity for
hyperk\"ahler quotients to the following conjecture.

\begin{conjecture}[\cite{CohomologyofHKQuotients}]
\label{FlowClosedConjecture}
If $\muc$ is a complex moment map associated to a linear action by a compact 
group $K$ on a vector space, then $|\muc|^2$ is flow-closed. 
\end{conjecture}

This was proved for the special case of $S^1$ actions in 
\cite{CohomologyofHKQuotients}, but the method of proof does not
admit any obvious generalization. We will prove the following.

\begin{theorem} Conjecture \ref{FlowClosedConjecture} is true when $K$
is a torus.
\end{theorem}

This is an immediate consequence of Proposition \ref{FlowClosedness}
and Theorem \ref{GlobalEstimate}.
To appreciate why this result requires some effort, let us contrast it with
the analogous statement for $|\mu|^2$, where $\mu$ is an ordinary moment
map associated to a linear action on a Hermitian vector space $V$. It 
is immediate from the definitions that $\nabla |\mu|^2 = 2 I v_{\mu}$.
Hence, the gradient trajectory through a point $x$ always remains in the
$K_\CC$ orbit through $x$. Thus it suffices to restrict attention to the
$K_\CC$ orbits. Flow-closedness is then a consequence of the  following
proposition.

\begin{proposition}[{Sjamaar \cite[Lemma 4.10]{SjamaarConvexity}}] 
\label{SjamaarTrick}
Suppose $K_\CC$ acts linearly
on $V$, let $\{g_n\}$ be a sequence of points in $K_\CC$, and 
let $\{x_n\}$ be a bounded sequence in $V$. Then the sequence
$\{g_n x_n\}$ is bounded if $\{\mu(g_n x_n)\}$ is bounded.
\end{proposition}

Now consider the gradient flow of $|\muc|^2$ on $T^\ast V$. Since 
$|\muc|^2 = |\mu_2|^2 + |\mu_3|^2$, we see that
$\nabla |\muc|^2 = 2 I_2 v_{\mu_2} + 2 I_3 v_{\mu_3}$.
The problem is immediate: due to the simultaneous appearance of $I_2$ and $I_3$,
the gradient trajectories appear to lie
on the orbits of a ``quaternification'' of $K$,
but in general, no suitable quaternification of $K$ exists.
More precisely, if we let  $\mathcal{D}$ be
the distribution on $V$ generated by the $K$ action
(i.e.\ the integrable distribution whose leaves are the $K$-orbits), then
$\mathcal{D}_\CC := \mathcal{D} + I \mathcal{D}$ is integrable,
whereas the distribution
$\mathcal{D}_\HH := \mathcal{D} + I_1 \mathcal{D} 
 + I_2 \mathcal{D} + I_3 \mathcal{D}$ on $T^\ast V$ is not integrable in general.

In a few specific examples, a detailed study of $\mathcal{D}_\HH$ leads to
definite conclusions about the gradient flow, but at present
we cannot prove a general result in this direction. In any case,
we will not pursue this approach in the present work (but see 
Remarks \ref{CrossTermRemark} and \ref{SubriemannianRemark}, which are related
to this problem).

With these considerations in mind, it would be illuminating to have a proof
of the flow-closedness of $|\mu|^2$ that relies neither on arguments
involving $K_\CC$ nor on properness, as such a proof might
generalize to the hyperk\"ahler setting. This is exactly the content
of our main result, Theorem \ref{GlobalEstimate}. The key idea is to
relax the assumption of properness to the weaker condition of satisfying
a certain gradient inequality, which we introduce next.
%---------------------------------------------------------------------

%---------------------------------------------------------------------
\section{\loj\ Inequalities}
\label{LojasiewiczInequalities}
%---------------------------------------------------------------------
We begin by giving a precise definition of the type of inequality we
wish to consider, as well as its most important consequence.

\begin{definition} Let $(M, g)$ be a complete Riemannian manifold. A smooth
real-valued function $f$ on $M$ is said to satisfy a
\emph{global \loj\ inequality} if for any real number $f_c$ in the closure of
the image of $f$, there exist constants $\epsilon > 0$, $k > 0$, and
$0 < \alpha < 1$, such that
\[ |\grad f(x)| \geq k|f(x) - f_c|^\alpha, \]
for all $x \in M$ such that $|f(x) - f_c| < \epsilon$.
\end{definition}

\begin{remark} The term global \loj\ inequality is borrowed from
\cite{JKSGlobalLojasiewicz}; however, we use it in a different way, as
we are concerned specifically with bounding the gradient of $f$.
\end{remark}

\begin{proposition} 
\label{FlowClosedness}
Suppose $f$ satisfies a \gli\ and is bounded below.
Then $f$ is flow-closed.
\end{proposition}
\begin{proof} Let $x(t)$ be a trajectory of $-\grad f$.
Since $f(x(t))$ is descreasing and bounded below,
$\lim_{t \to \infty} f(x(t))$ exists. Call this limit $f_c$.
Let $\epsilon, k,$ and $\alpha$ be the constants appearing in the global
\loj\ inequality for $f$ with limit $f_c$.
For large enough $T$, we have $|f(x(t)) - f_c| < \epsilon$ 
whenever $t > T$. Consider $t_2 > t_1 > T$, and let $f_1 = f(x(t_1))$ and 
$f_2 = f(x(t_2))$. Since $\dot{x} = - \grad f$, we have that
\[ d(x(t_1), x(t_2)) \leq \int_{t_1}^{t_2} |\grad f(x(t))| dt, \]
where $d(x,y)$ denotes the Riemannian distance.
By the change of variables $t \mapsto f(x(t))$, we obtain
\begin{eqnarray*}
d(x(t_1), x(t_2)) &\leq& \int_{f_2}^{f_1} |\grad f|^{-1} df \\
&\leq& k^{-1} \int_{f_2}^{f_1} |f - f_c|^{-\alpha} df \\
&=& k^{-1} (1-\alpha)^{-1} \left(|f_1-f_c|^{1-\alpha} - |f_2-f_c|^{1-\alpha} \right)\\
&<& k^{-1} (1-\alpha)^{-1} |f(x(T)) - f_c|^{1-\alpha}.
\end{eqnarray*} 
Since $\alpha < 1$, the last expression can be made arbitrarily small by
taking $T$ sufficiently large, so we see that $\lim_{t \to \infty} x(t)$
exists,
and in particular the gradient trajectory is contained in a compact set.
\end{proof}

This argument establishes flow-closedness directly from the
{\L}ojasiewicz inequality, without appealing to compactness.
Thus it would be sufficient to show that $|\mu|^2$ satisfies such an inequality.
To motivate why we might expect this to be case,
we recall the classical \loj\ inequality.
\begin{theorem}[\loj\ Inequality \cite{BierstoneMilman,Lojasiewicz}]
\label{LojasiewiczInequality} 
Let $f$ be a real analytic function on an open set $U \in \mathbb{R}^N$, and
let $c$ be a critical point of $f$. Then on any compact set $K \subset U$,
there are constants $k > 0$ and $0 < \alpha < 1$ such that the inequality
\[ \norm{\nabla f(x)} \geq k \left|f(x) - f(c)\right|^\alpha \]
holds for all $x \in K$.
\end{theorem}

If $f$ is a \emph{proper} real analytic function, then this immediately
implies that $f$ satisfies a global \loj\ inequality as defined above.
Since our primary concern is Morse theory, we can relax the assumption
of analyticity as follows.
\begin{proposition} Suppose $f$ is a proper Morse function. Then $f$ 
satisfies a global \loj\ inequality.
\end{proposition}
\begin{proof} By the Morse lemma, near each critical point we can choose
coordinates in which $f$ is real analytic. Hence $f$ satisfies the classical
\loj\ inequality near each critical point, and since $f$ is proper this can
be extended to a global inequality.
\end{proof}

For a moment map $\mu$, the function $|\mu|^2$ is in general
neither Morse nor Morse-Bott, but minimally degenerate
(in the sense of \cite{KirwanCohomology}). Nonetheless, we can still
obtain a global \loj\ inequality whenever $\mu$ is proper.
\begin{proposition} Suppose $\mu$ is a moment map associated to an action 
of a compact Lie group, and suppose furthermore that $\mu$ is proper.
Then $|\mu|^2$ satisfies a \gli.
\end{proposition}
\begin{proof} 
Lerman \cite{LermanGradientFlow} uses local normal forms to show
that $|\mu|^2$ is real analytic in a neighborhood of each component of
its critical set, and hence satisfies the classical \loj\ inequality.
Since $\mu$ is proper, this can be extended to a global inequality.
\end{proof}

Since we would like to drop the assumption of properness,  it is natural to ask
whether there are examples of moment maps which are \emph{not} proper
but nevertheless satisfy a global \loj\ inequality. 
The answer to this question is in the affirmative, at least when the action 
is linear. The following is our main theorem, which we prove in 
Section \ref{Proof}.
\begin{theorem} 
\label{GlobalEstimate}
Let $\mu$ be a moment map associated to a unitary representation
of a compact group $K$ on a Hermitian vector space $V$.
Then $f = |\mu|^2$ satisfies a 
\gli. In detail, for every $f_c \geq 0$, there exist constants 
$k > 0$ and $\epsilon > 0$ such that
\begin{equation} \label{GlobalInequality}
|\nabla f(x)| \geq k|f(x) - f_c|^\frac{3}{4}
\end{equation}
whenever $|f(x) - f_c| < \epsilon$.
If $K$ is a torus, then the same holds for the functions $|\muc|^2$
and $|\muhk|^2$ associated to the action of $K$ on $T^\ast V$.
\end{theorem}

\begin{remark} \label{FirstNeemanRemark}
This theorem is a generalization of
\cite[Theorem A.1]{NeemanQuotientVarieties}.
However, in \cite{NeemanQuotientVarieties}, it is assumed that the 
constant term in the moment map is chosen so that $f$ is homogeneous
(see equation \ref{OrdinaryMomentMapFormula}),
leading to an inequality of the form
\[ |\grad f| \geq k f^\frac{3}{4}, \]
which holds on all of $V$. Both sides are homogeneous of the same degree,
so that it suffices to prove the inequality on the unit sphere, allowing
the use of compactness arguments. In Theorem \ref{GlobalEstimate}, we
make no such assumption, and this complicates several steps of the proof.
\end{remark}

\begin{remark} If $X \subset V$ is a complex subvariety, then it is easy to
see that Theorem \ref{GlobalEstimate} implies that the restriction
of $f$ to $X$ satisfies a global \loj\ inequality. Similarly,
if $X \subset T^\ast V$ is a hyperk\"ahler subvariety, we deduce the inequality
for the restrictions of $|\muc|^2$ and $|\muhk|^2$.
\end{remark}

\begin{remark} In \cite{GradientSemialgebraic} and \cite{KurdykaGradients} it
is shown that certain classes of functions satisfy similar global {\L}ojasiewicz 
inequalities; indeed this was the motivation to consider such inequalities. 
However, these general theorems cannot rule out the possibility $\alpha \geq 1$,
which is not sharp enough to prove the boundedness of all gradient
trajectories. In this sense, the real content of 
Theorem \ref{GlobalEstimate} is the bound on the exponent. 
\end{remark}

%\begin{remark}
%{\L}ojasiewicz-type inequalities have also appeared in gauge-theoretic contexts
%in the form of curvature estimates. See for example Proposition 3.5 of 
%\cite{WilkinHiggs}.
%\textbf{Uhlenbeck compactness as well (related to Sjamaar's argument)}.
%\end{remark}
%---------------------------------------------------------------------

%---------------------------------------------------------------------
In what follows,
let $K$ be a compact group acting unitarily on a Hermitian vector space $V$,
with moment map $\mu$. Note that since the action is linear, we have
$H_K^\ast(V) = H_K^\ast(\mathrm{point}) =: H_K^\ast$. For $\alpha \in \fkd$
and $\beta \in \fkd_\CC$, we denote
\begin{eqnarray}
X(\alpha) &:=& V \reda{\alpha} K, \\
M(\alpha, \beta) &:=& T^\ast V \rreda{(\alpha, \beta)} K.
\end{eqnarray}

\begin{corollary}
\label{DeformationRetract} Let $f = |\mu|^2$.
Then for each component $C$ of the 
critical set of $f$, the gradient flow defines a $K$-equivariant homotopy from 
the stable manifold $S_C$ to the critical set $C$. If $K$ is abelian, then
the same holds for the functions $|\muc|^2$ and $|\muhk|^2$.
\end{corollary}

\begin{remark}
If $\mu$ is assumed to be proper then this is a special case of the main 
theorem of \cite{LermanGradientFlow};
in the present case we make no such assumption. However, the essential
ingredient of the proof is the \loj\ inequality.
We give the proof below for completeness, but the details
do not differ significantly from \cite{LermanGradientFlow}.
\end{remark}

\begin{proof}
We define a continuous map $F: S_C \times [0, \infty) \to S_C$ by
$(x, t) \mapsto x(t)$, where $x(t)$ is the trajectory of $-\grad f$ beginning
at $x$, evaluated at time $t$.
By Proposition \ref{FlowClosedness}, we can extend
this to a map $F: S_C \times [0, \infty] \to S_C$ by
$(x, \infty) \mapsto \lim_{t \to \infty} x(t)$. This map is the identity
when restricted to $C$, and maps $S_C \times \{\infty\}$ to $C$.
We must verify that this extended map is continuous.

We must show that for any $x_0 \in S_C$ and any sequence $\{(x_n, t_n)\}$ of 
points in $S_C \times [0, \infty]$ satisfying 
$\lim_{n\to\infty} x_n = x_0 \in S_C$ and
$\lim_{n\to\infty} t_n = \infty$ that $\lim_{n\to\infty} x_n(t_n)$ exists.
We will show that it
is equal to $x_c := \lim_{t\to\infty} x_0(t)$. 
Let $f_c$ be the value of $f$ on $C$ and let $\epsilon, k$ be the 
constants appearing in Theorem \ref{GlobalEstimate}.
Given such a sequence,
%\begin{itemize}[label=\textperiodcentered,nolistsep]
%\item let $\eta > 0$ and assume $\eta < \epsilon$ but is otherwise arbitrary,
%\item let $T > 0$ be chosen large enough so that $|f(x_0(t)) - f_c| < \eta$
%for $t > T$, and
%\item let $N > 0$ be chosen to that $t_n > T$ for $n > N$.
%\end{itemize}
let $\eta > 0$ and assume $\eta < \epsilon$ but is otherwise arbitrary;
let $T > 0$ be chosen large enough so that $|f(x_0(t)) - f_c| < \eta$
for $t > T$; and let $N > 0$ be chosen to that $t_n > T$ for $n > N$.
The map $S_C \to S_C$ given by $x \mapsto x(T)$ is continuous, 
so we can find $\delta > 0$ such that
\[ |x - x_0| < \delta \implies |x(T) - x_0(T)| < \eta. \]
Since $x \mapsto f(x(T))$ is continuous, we can shrink $\delta$ if 
necessary so that
\[ |x - x_0| < \delta \implies |f(x(T)) - f(x_0(T))| < \eta. \]
Choose $N$ larger if necessary such that $|x_n - x_0| < \delta$ for $n > N$.
We would like to show that $|x_n(t_n) - x_c| \to 0$ as $n \to \infty$.
For $n > N$ we have
\begin{eqnarray*}
|x_n(t_n) - x_c|  &=& |x_n(t_n) - x_n(T) + x_n(T) - x_0(T) + x_0(T) - x_c| \\
&\leq& |x_n(t_n) - x_n(T)| + |x_n(T) - x_0(T)| + |x_0(T) - x_c|.
\end{eqnarray*}
By our choice of $N$, the second term is bounded by $\eta$, and 
we may apply the argument in the proof of Proposition \ref{FlowClosedness}
to show that the third term is bounded by
$k^{-1} |f(x_0(T)) - f_c|^\frac{1}{4} < 4k^{-1} \eta^\frac{1}{4}$.
Finally, to bound the first term we again apply the argument of Proposition
\ref{FlowClosedness} to obtain the bound
\[ 4k^{-1} \left(|f(x_n(T)) - f_c|^\frac{1}{4} 
 - |f(x_n(t_n)) - f_c|^\frac{1}{4} \right) < 
 4k^{-1} |f(x_n(T)) - f_c|^\frac{1}{4}. \]
Since $|x_n - x_0| < \delta$, we have
\begin{eqnarray*}
|f(x_n(T)) - f_c| &=& |f(x_n(T)) - f(x_0(T)) + f(x_0(T)) - f_c| \\
&\leq& |f(x_n(T)) - f(x_0(T))| + |f(x_0(T)) - f_c| \\
&<& 2 \eta.
\end{eqnarray*}
Hence the first term is bounded by 
$4 k^{-1} (2\eta)^\frac{1}{4} < 8k^{-1} \eta^\frac{1}{4}$, and we obtain
\[ |x_n(t_n) - x_c| \leq 12k^{-1} \eta^\frac{1}{4} + \eta. \]
Since we may take $\eta$ arbitrarily small, we see that
$|x_n(t_n) - x_c| \to 0$.
\end{proof}

\begin{corollary} \label{HomotopyCorollary}
If $\beta \in \fk_\CC^\ast$ is regular central, then for any central
$\alpha \in \fkd$, the set $\mur^{-1}(\alpha) \cap \muc^{-1}(\beta)$ is
a $K$-equivariant deformation retract of $\muc^{-1}(\beta)$, and
in particular
\[ H_K^\ast(\muc^{-1}(\beta)) \iso H^\ast( M(\alpha, \beta) ). \]
\end{corollary}
\begin{proof} Since $\beta$ is regular central, $\muc^{-1}(\beta)$ is
a $K$-invariant complex submanifold of $T^\ast V$, and furthermore $K$
acts on $\muc^{-1}(\beta)$ with at most discrete stabilizers. Hence the
only component of the critical set of $|\mur-\alpha|^2$ that intersects
$\muc^{-1}(\beta)$ is the absolute minimum, which occurs on $\mur^{-1}(\alpha)$.
By Corollary \ref{DeformationRetract}, the gradient flow of $-|\mur-\alpha|^2$
gives the desired $K$-equivariant deformation retract from $\muc^{-1}(\beta)$ to
$\mur^{-1}(\alpha) \cap \muc^{-1}(\beta)$. Thus
\[ H_K^\ast(\muc^{-1}(\beta)) 
\iso H_K^\ast(\mur^{-1}(\alpha) \cap \muc^{-1}(\beta))
\iso H^\ast( M(\alpha, \beta) ). \]
\end{proof}

\begin{corollary} \label{SurjectivityForTori}
If $\alpha$ is a regular central value of $\mu$, then the
Kirwan map $H_K^\ast \to H^\ast(X(\alpha))$ is surjective. If $K$ is
a torus and $(\alpha, \beta)$ is a regular value of the hyperk\"ahler
 moment map,
then the hyperk\"ahler Kirwan map $H_K^\ast \to H^\ast(M(\alpha,\beta))$
is surjective.
\end{corollary}
\begin{proof} Theorem \ref{GlobalEstimate} and Proposition \ref{FlowClosedness}
show that $|\mu-\alpha|^2$ is flow-closed, so we obtain surjectivity of 
$H_K^\ast \to H^\ast(X(\alpha))$ by
Theorem \ref{KirwanSurjectivity}.

In the hyperk\"ahler case, if $K$ is a torus
then $|\muc-\beta|^2$ is flow-closed, so by Theorem \ref{SurjectivityCriterion}
we obtain surjectivity of $H_K^\ast \to H^\ast_K(\muc^{-1}(\beta))$.
By Corollory \ref{HomotopyCorollary}, the map
$H_K^\ast(\muc^{-1}(\beta)) \to H^\ast(M(\alpha, \beta))$ is an
isomorphism, so we have surjectivity of $H_K^\ast \to H^\ast(M(\alpha, \beta))$.
\end{proof} 

\begin{remark}
Konno proved surjectivity of the map 
$H_K^\ast \to H^\ast(M(\alpha,\beta))$ when $K$ is a
torus using rather different means \cite{KonnoHypertoric}.
Konno also computed the kernel of the Kirwan map, giving an explicit description
of the cohomology ring $H^\ast(M(\alpha, \beta))$. We will study this case in 
detail in Section \ref{HypertoricSection}, and we will see that Morse theory
allows us to compute the kernel very easily. 
\end{remark}
%---------------------------------------------------------------------

%---------------------------------------------------------------------
\section{Proof of the Main Theorem} 
\label{Proof}
%---------------------------------------------------------------------
Let $K$ be a compact Lie group with Lie algebra $\fk$, and
suppose $K$ acts unitarily on a Hermitian vector space $V$. 
Without loss of generality we will regard $K$ as a subgroup of $U(V)$
and identify $\fk$ with a Lie subalgebra of $\mathfrak{u}(V)$,
which we identify with the Lie algebra of skew-adjoint matrices. 
We will use the trace norm on $\mathfrak{u}(V)$ to induce an invariant
inner product on $\fk$, and use this to identify $\fk$ with its dual.
 For any $\xi \in \fk$ there is
a fundamental vector field $v_\xi$ which is given by $v_\xi(x) = \xi x$.
We denote by $\stab(x)$ the Lie algebra of the stabilizer of a point
$x \in V$; i.e.
\begin{equation}
\stab(x) = \{ \xi \in \fk \suchthat v_\xi(x) = 0\}.
\end{equation}
If we fix an orthonormal basis $\{e_a\}$ of $\fk$, then a moment map
is given by
\begin{equation} \label{OrdinaryMomentMapFormula}
\mu(x) = \sum_a \frac{1}{2} \inner{i e_a x}{x} - \alpha,
\end{equation}
where $i = \sqrt{-1}$ is the complex structure on $V$ and $\alpha$
is any central element of $\fk$. Then for $f = |\mu|^2$,
we have
\begin{equation}
\label{GradFormula}
\nabla f(x) = 2i v_{\mu(x)}(x) = \sum_a 2 i \mu^a(x) e_a x,
\end{equation}
and since $i$ is unitary, we have that
\begin{equation}
\label{GradEquality}
|\nabla f| = 2|v_\mu|. 
\end{equation}

\begin{lemma} \label{CrossTermVanish}
Suppose $K$ is abelian, and consider its action on
$T^\ast V$. Let $f_i = |\mu_i|^2$ for $i = 1,2,3$. Then
$\inner{\nabla f_i}{\nabla f_j} = 0$ for $i \neq j$.
\end{lemma}
\begin{proof} We compute:
\begin{eqnarray*}
\inner{\nabla f_2}{\nabla f_3} 
&=& 4 \inner{J v_{\mu_2}}{K v_{\mu_3}} \\
&=& 4 \inner{I v_{\mu_2}}{v_{\mu_3}} \\
&=& 4 \omega_1(v_{\mu_2}, v_{\mu_3}) \\
&=& 4 \inner{\mu_2}{d\mu_1(v_{\mu_3})} \\
&=& 4 \inner{\mu_2}{[\mu_3, \mu_1]} \\
&=& 0.
\end{eqnarray*}
Similar computations show that the other two cross terms vanish.
\end{proof}
\begin{remark} \label{CrossTermRemark}
The proof of this lemma makes it clear why the assumption
that $K$ is abelian is so useful in the hyperk\"ahler setting.
The cross term 
\[ \inner{\nabla f_j}{\nabla f_k} = \pm 4\inner{\mu_1}{[\mu_2, \mu_3]} \]
is exactly the obstruction to proving an estimate in the general nonabelian case.
The function on the right hand side is very natural, and seems to be
genuinely hyperk\"ahler, having no analogue in
symplectic geometry. Numerical experiments suggest
that it is small in magnitude compared to $|\nabla f_2| + |\nabla f_3|$,
but we do not know how to prove this. A theorem in this direction might
be enough to prove flow-closedness (and hence Kirwan surjectivity) in 
general. It certainly warrants further study.
\end{remark}
\begin{remark} \label{SubriemannianRemark}
In light of the discussion following Proposition
\ref{SjamaarTrick}, the cross term $\inner{\mu_1}{[\mu_2, \mu_3]}$
should have an interpretation in terms of the geometry of the 
non-integrable distribution $\mathcal{D}_\HH$.
Subriemannian geometry may have a key role to play in proving
Kirwan surjectivity for hyperk\"ahler quotients by nonabelian groups.
\end{remark}

\begin{proposition} 
\label{ReduceToTorus} 
Let $T \subset K$ be a maximal torus of $K$, and let $\mu_K$ and $\mu_T$ be
the corresponding moment maps. Let $f_K = |\mu_K|^2$ and $f_T = |\mu_T|^2$.
Suppose that for any $f_c \geq 0$, $f_T$ satisfies a global \loj\ inequality. 
Then $f_K$ satisfies a global \loj\ inequality with the same constants and 
exponent.
\end{proposition}
\begin{proof}
Since $K$ is compact, for each
$x \in V$ we can find some $k \in K$ so that $\Ad_{k} \mu_K(x) \in \ft$. 
By equivariance of the moment map, we have 
$\Ad_{k} \mu_K(x) = \mu_K(k x) \in \ft$.
Hence $\mu_K(k x) = \mu_T(k x)$, so that
$|v_{\mu_K(kx)}| = |v_{\mu_T(kx)}|$.
Using equality (\ref{GradEquality}), this tells us that
$|\nabla f_K(kx)| = |\nabla f_T(kx)|$. 
Since $f_K(kx) = |\mu_K(kx)|^2 = |\mu_T(kx)|^2 = f_T(kx)$,
we deduce the \loj\ 
inequality for $f_K$ from the inequality for $f_T$.
\end{proof}

We assume for the remainder of this section that $K$ is a torus.

\begin{proposition}
\label{StrongerCondition} 
Fix $f_c \geq 0$, and suppose that for each 
$\mu_c \in \fk$ satisfying $|\mu_c|^2 = f_c$, there exist
constants $\epsilon' > 0$ and $c' > 0$ (depending on $\mu_c$) such that 
\begin{equation} \label{GlobalInequalityWeaker}
|\nabla f(x)| \geq c'|f(x) - f_c|^\frac{3}{4}
\end{equation}
whenever $|\mu(x) - \mu_c| < \epsilon'$. Then $f$ satisfies a global
\loj\ inequality, i.e., there exist constants $\epsilon > 0$ and
$c > 0$ so that the inequality (\ref{GlobalInequalityWeaker}) holds whenever 
$|f(x) - f_c| < \epsilon$. 
\end{proposition}
\begin{proof} Suppose that for each $\mu_c$ as above we can find
constants $\epsilon(\mu_c)$ and $c(\mu_c)$ so that inequality
(\ref{GlobalInequalityWeaker}) holds. Let
$U(\mu_c)$ be the $\epsilon(\mu_c)$-ball in $\fk$ centered at $\mu_c$.
These open sets cover the sphere $S$ of radius $\sqrt{f_c}$ in $\fk$, and by
compactness we can choose a finite subcover. Denote this finite subcover by
$\{U_i\}_{i=1}^n$, with centers $\mu_i$ and constants $c_i$, and
let $c = \min_i c_i$. The finite union $\cup_i U_i$ 
contains an $\epsilon$-neighbourhood of $S$ for some sufficiently small
$\epsilon$.  Since $f(x) = |\mu(x)|^2$, if we choose
$\epsilon' > 0$ sufficiently small then $|f(x) - f_c| < \epsilon'$ implies
that $\left| |\mu(x)| - \sqrt{f_c} \right| < \epsilon$, so that 
$\mu(x) \in \cup_i U_i$. In particular, there is some $j$ such that
$\mu(x) \in U_j$, and by inequality (\ref{GlobalInequalityWeaker}) we have
\[ |\nabla f(x)| \geq c_j |f(x) - f_c|^\frac{3}{4} 
\geq c|f(x) - f_c|^\frac{3}{4}, \]
as desired.
\end{proof} 

Before giving the proof of Theorem \ref{GlobalEstimate}, we isolate some of
the main steps in the following lemmas.
Let us introduce the following notation. For $x \in V \setminus \{0\}$,
let $\hat{x}$ denote its projection to the unit sphere, i.e.\
$\hat{x} = x/|x|$ or equivalently $x = |x|\hat{x}$.
Since the action is linear, we have that $v_\xi(x) = |x| v_\xi(\hat{x})$
and $\stab(x) = \stab(\hat{x})$.
 
\begin{lemma} 
\label{KeyEstimate1} 
Fix $\hat{y}$ 
in the unit sphere in $V$. Let $P$ be the orthogonal projection from $\fk$ to 
$\stab(\hat{y})^\perp$, and $Q = 1-P$. Then there is a neighbourhood 
$U$ of $\hat{y}$ such that for any $\xi \in \fk$, inequalities
\begin{eqnarray}
\label{KeyInequality1}
 |v_\xi(x)| &\geq& c |x| | P \xi | \\
\label{KeyInequality2}
 |v_\xi(x)| &\geq& c' |v_{P \xi}(x)| \\
\label{KeyInequality3}
 |v_\xi(x)| &\geq& c'' \left(|v_{P\xi}(x)| + |v_{Q\xi}(x)| \right)
\end{eqnarray}
hold for all $x$ such that $\hat{x} \in U$.
The constants $c,c',c''$ are 
positive and depend only on $\hat{y}$ and $U$ but not on $x$ or $\xi$.
\end{lemma}
\begin{remark} A version of this lemma appears as part of the proof of
\cite[Theorem A.1]{NeemanQuotientVarieties}, though it is not stated
exactly as above. We repeat the argument below so that our proof of Theorem
\ref{GlobalEstimate} is self-contained.
\end{remark}
\begin{proof}
Fix $\hat{y}$ and let $P$ and $Q$ be as above. 
Let $W$ be the smallest $K$ invariant subspace of $V$ containing $\hat{y}$,
and let $P_W: V \to W$ be the orthogonal projection. Note that $W$ is
generated by vectors of the form $\xi_1 \cdots \xi_l \hat{y}$, with
$\xi_i \in \fkc$. Since $P_W$ is a projection, $|v_\xi| \geq |P_W v_\xi|$, 
so to establish inequality 
(\ref{KeyInequality1}) it suffices to show that 
$|P_W v_\xi| \geq c|P \xi|$.
Note that $P_W$
is equivariant, i.e.\ $\xi P_W = P_W \xi$ for all $\xi \in \fk$. Note also
that since $K$ is abelian, if $\xi \in \stab(\hat{y})$ then $\xi \in \ann(W)$,
since $\xi \xi_1 \cdots \xi_l \hat{y} = \xi_1 \cdots \xi_l \xi \hat{y} = 0$.
For any orthonormal basis $\{e_a\}_{i=1}^d$ of $\fk$ chosen so that
$\{e_a\}_{i=1}^n$ is an orthonormal basis of $\stab(\hat{y})^\perp$ and
$\{e_a\}_{i=n+1}^d$ is an orthonormal basis of $\stab(\hat{y})$, we have
$P\xi = \sum_{a=1}^d \xi^a Pe_a = \sum_{a=1}^n \xi^a e_a$.
Similarly, we find
\[ P_W v_\xi(x) = \sum_{a=1}^d P_W \xi^a e_a x 
= \sum_{a=1}^d \xi^a e_a P_W x
= \sum_{a=1}^n \xi^a e_a P_W x = P_W v_{P\xi}(x). \]
Taking norms, we see that
\[ |P_W v_\xi(x)|^2 
= \sum_{a=1}^n \sum_{b=1}^n \xi^a \xi^b \inner{P_W e_a x}{P_W e_b x}
 = (P\xi)^T G(x) (P\xi), \]
where $G(x)$ is the matrix with entries 
$G_{ab}(x)  = \inner{P_W e_a x}{P_W e_b x}$ for $a,b = 1, \ldots, n$.
By construction,
this matrix is is positive definite at $\hat{y}$, so for a sufficiently 
small neighbourhood $U$ of $\hat{y}$, we obtain
$|P_W v_\xi(\hat{x})|^2 \geq c |P\xi|^2$,
with the constant $c$ depending only on $\hat{y}$ and the choice of
neighbourhood $U$. For any $x$ with $\hat{x} \in U$, we obtain
$|v_\xi(x)| = |x| |v_\xi(\hat{x})| \geq c |x| |P\xi|$,
which is inequality (\ref{KeyInequality1}).

We can deduce inequality (\ref{KeyInequality2}) from inequality 
(\ref{KeyInequality1}) as follows. We have
\[ |v_{P\xi}(\hat{x})| = |\sum_{a=1}^n \xi^a e_a \hat{x}|
 \leq |P\xi|\sum_{a=1}^n |e_a \hat{x}|. \]
Shrinking $U$ if necessary, we can assume that the functions
$|e_a \hat{x}|$ are bounded on $U$, and so we obtain
$|v_{P\xi}(\hat{x})| \leq c'|P\xi| \leq cc' |v_\xi(\hat{x})|$.
Both sides are homogeneous of the same degree in $\hat{x}$, so the inequality 
holds for any $x$ with $\hat{x} \in U$. This establishes 
inequality (\ref{KeyInequality2}).

To establish inequality (\ref{KeyInequality3}), first note the following
consequence of the triangle inequality. If $v,w$ are vectors in some
normed vector space, and $|v + w| \geq a |v|$ with $a > 0$, then we have
\begin{equation} \label{MinorUsefulInequality}
 |v| + |w| = |v| + |v+w - v| \leq 2|v| + |v+w| 
\leq \left(1 + \frac{2}{a} \right)|v+w|.
\end{equation}
Since $\xi = P\xi + Q\xi$, we have
$v_\xi(x) = v_{P\xi}(x) + v_{Q\xi}(x)$, so applying inequality
(\ref{KeyInequality2}) together with the inequality
(\ref{MinorUsefulInequality}) above, we obtain
\[ |v_\xi(x)| \geq c''' \left( |v_{P\xi}(x)| + |v_{Q\xi}(x)| \right), \]
with the constant $c'''$ depending only on $\hat{y}$ and the neighbourhood
$U$.
\end{proof}

\begin{lemma}
\label{KeyEstimate2}
Let $f = |\mu|^2$ and fix $\mu_c \in \fk$. Then there exist
constants $c > 0$ and $\epsilon > 0$ (depending on $\mu_c$) such that whenever
$|\mu(x) - \mu_c| < \epsilon$, we have 
\begin{equation} \label{KeyInequality4}
|x|^2 |f(x)| \geq c \left|f(x) - f_c \right|^\frac{3}{2}.
\end{equation}
\end{lemma}
\begin{proof}
Fix some particular $\mu_c$. Recall that $\mu$ is quadratic in the coordinate 
$x$ with no linear terms. Thus $\mu$ is affine in the coordinates
$v_{ij} = x_i \bar{x}_j$, and we may write
$\mu(x) = \phi(v)$, where $\phi(v) = Av - \alpha$, for some linear 
transformation $A: V \otimes V \to \fk$. We have
$|v_{ij}| = |x_i||x_j| \leq \frac{1}{2}\left(|x_i|^2 + |x_j|^2 \right)$,
so that
$|v| \leq \frac{1}{2} \sum_{i,j} |x_i|^2 + |x_j|^2 = N |x|^2$,
where $N = \dim V$. Thus
$|x|^2 |f(x)| = |x|^2 |\phi(v)|^2 \geq N^{-1} |v| |\phi(v)|^2$,
hence it suffices to show that
\[ |v| |\phi(v)|^2 \geq c \left| |\phi(v)|^2 - f_c \right|^\frac{3}{2}, \]
whenever $|\phi(v) - \mu_c| < \epsilon$. This follows immediately
from Lemma \ref{RhoPhiLemma}, which we state and prove below.
\end{proof}
  
\begin{lemma}
\label{RhoPhiLemma} Let $V_1$ and $V_2$ be inner product spaces, and consider
an affine map $\phi: V_1 \to V_2$ given by $\phi(v) = Av - \alpha$
for some linear map $A: V_1 \to V_2$ and constant $\alpha \in V_2$. 
Then for any $\phi_c$ in the image of $\phi$, there exist constants $c > 0$ and 
$\epsilon > 0$ such that
\[ |v| |\phi(v)|^2 \geq c \npvz^\frac{3}{2} \]
whenever $|\psi(v) - \phi_c| < \epsilon$.
\end{lemma}
\begin{proof} To avoid unnecessary clutter, we sometimes use
the shorthand $v^2 := |v|^2$ below.
First note that if $P$ is a projection such that $AP = A$,
then we have $\phi(v) = \phi(Pv)$, and so
\[ |v| |\phi(v)|^2 = |v| |\phi(Pv)|^2 \geq |Pv| |\phi(Pv)|^2. \]
Taking $P$ to be the orthogonal projection from $V_1$ to $(\ker A)^\perp$, 
without loss of generality we can assume that $A$ is injective. Similarly, 
without loss of generality we may assume that $A$ is surjective.
Suppose that $\phi_c \in V_2$ is fixed. Pick $v_c \in V_1$ so that
$\phi(v_c) = \phi_c$. There are four possible cases.

Case 1: $v_c = 0, \phi_c = 0$. In this case, $\alpha = 0$, and
$\phi(v) = Av$, so that $|\phi(v)| \leq |A| |v|$. Thus
$|v| |\phi(v)|^2 \geq |A|^{-1} |\phi(v)|^3$,
as desired. We may take $\epsilon$ to be any positive number, and
$c = |A|^{-1}$.

Case 2: $v_c = 0, \phi_c \neq 0$. Take $\epsilon \leq |\phi_c|/2$.
Then
\begin{eqnarray*}
|\phi(v)^2 - \phi_c^2|
&=& |\inner{\phi(v)-\phi_c}{\phi(v)+\phi_c}| \\
&\leq& |\phi(v) - \phi_c| |\phi(v) + \phi_c | \\
&\leq& |\phi(v) - \phi_c| \left( 2|\phi_c| + \epsilon \right) \\
&\leq& \frac{5}{2} |\phi(v) - \phi_c| |\phi_c|,
\end{eqnarray*}
so we have
$|\phi(v)^2 - \phi_c^2| \leq c_1 |\phi(v) - \phi_c| |\phi_c|$,
where $c_1$ is a numerical constant independent of $\phi_c$.
Since $|\phi(v) - \phi_c| < |\phi_c|/2$, we also have
$|\phi(v)^2 - \phi_c^2| \leq c_2 |\phi_c|^2$.
Combining these two inequalities, we have
$|\phi(v)^2 - \phi_c^2|^\frac{3}{2} \leq c_3 |\phi(v) - \phi_c| |\phi_c|^2$.
Since $\phi(v) - \phi_c = A(v - v_c) = Av$, we have
$|\phi(v) - \phi_c| \leq |A| |v|$. Putting this back into the previous
inequality, we obtain
$|\phi(v)^2 - \phi_c^2|^\frac{3}{2} \leq |A| |v| |\phi_c|^2$.
On the other hand, with our choice of $\epsilon$ we have
$|\phi| \geq \frac{1}{2} |\phi_c|$, so that
$|\phi(v)^2 - \phi_c^2|^\frac{3}{2} \leq 4 |A| |v| |\phi|^2$, 
as desired.

Case 3: $v_c \neq 0, \phi_c = 0$. Take 
$\epsilon = |A^{-1}|^{-1} |A|^{-1} |\alpha| / 2$.
Since in this case $A v_c = \alpha$, we have $\epsilon \leq |A^{-1}|^{-1}|v_c|/2$,
 and
\begin{eqnarray*}
|v_c| &=& |v - (v-v_c)| \\
&\leq& |v| + |v-v_c| \\
&=& |v| + |A^{-1} \phi(v)| \\
&\leq& |v| + |A^{-1}| \epsilon \\
&\leq& |v| + \frac{|v_c|}{2}.
\end{eqnarray*}
Thus $|v| \geq |v_c|/2 \geq |A|^{-1} |\alpha|/2$. Then
\[ |\phi(v)|^3 \leq \epsilon |\phi(v)|^2 
= \frac{1}{2} |A^{-1}|^{-1}|A|^{-1} |\alpha| |\phi(v)|^2
\leq |A^{-1}|^{-1} |v| |\phi(v)|^2, \]
which is the desired inequality.

Case 4: $v_c \neq 0, \phi_c \neq 0$. Let $\epsilon'$ be chosen as in case (3),
and let $\epsilon = \min\{\epsilon', |\phi_c|/2\}$. As in the previous cases,
this choice of $\epsilon$ guarantees that $|v| \geq |v_c|/2$, and that
$|\phi(v)| \geq |\phi_c|/2$. As before,
\begin{eqnarray*}
|\phi(v)^2 - \phi_c^2| &\leq& |\phi(v) - \phi_c| |\phi_v + \phi_c| \\
&\leq& \epsilon ( 2 |\phi_c| + \epsilon) \\
&\leq& \frac{5}{4} |\phi_c|^2. 
\end{eqnarray*}
Similarly, since $|v-v_c| \leq |v_c|/2$ and 
$|\phi(v) - \phi_c| \leq |A||v - v_c|$, we also have
\[|\phi(v)^2 - \phi_c^2| \leq (3/4) |A| |v_c| |\phi_c|.\]
Putting these together, we have that
\[ |\phi(v)^2 - \phi_c^2|^\frac{3}{2} 
\leq c|v_c| |\phi_c|^2 \leq c' |v| |\phi(v)|^2, \]
where $c$ and $c'$ are numerical constants independent of $\phi_c$.
\end{proof}

Lemmas \ref{KeyEstimate1} and
\ref{KeyEstimate2} allow us to prove the following local estimates, which
are essential in the proof of Theorem \ref{GlobalEstimate}.

\begin{proposition} 
\label{FreeCase} 
Let $\mu_c \in \fk$ be fixed and $y$ is some point in $V$ with discrete 
stabilizer. Then there exists an open neighbourhood $U$ of $\hat{y}$ and 
constants $c > 0$ and $\epsilon > 0$ such that
\[ |\nabla f(x)|^2 \geq k|f(x) - f_c|^\frac{3}{2} \]
for all $x \in V \setminus \{0\}$ such that $\hat{x} \in U$ and
$|\mu(x) - \mu_c| < \epsilon$.
\end{proposition}
\begin{proof}
Suppose $y \in V$ and $\stab(y) = 0$. Then by
Lemma \ref{KeyEstimate1}, there is a neighbourhood $U$ of $\hat{y}$ so that
\[ |v_\xi(x)|^2 \geq c |x|^2 |\xi|^2 \]
for all $x$ such that $\hat{x} \in U$. Take $\xi = \mu(x)$ and apply
Lemma \ref{KeyEstimate2} to find $c'$ and $\epsilon$ so that
\[ |v_{\mu(x)}(x)|^2 \geq c' \left| |\mu(x)|^2 - f_c \right|^\frac{3}{2}, \]
whenever $|\mu(x) - \mu_c| < \epsilon$.
Since $|\nabla f(x)| = 2|v_{\mu(x)}(x)|$, this gives the desired inequality.
\end{proof}

\begin{proposition} 
\label{InductionStep}
Let $y \in V$ and suppose $\stab(y)$ is a proper nontrivial subspace of $\fk$. 
Then there are proper
subtori $K_1, K_2$ of $K$ such that $K \iso K_1 \times K_2$, a neighbourhood 
$U$ of $\hat{y}$, and a constant $c > 0$ such that
\[ |\nabla f(x)| \geq c\left( |\nabla f_{K_1}(x)| + |\nabla f_{K_2}(x)| \right) \]
for all $x \in V \setminus \{0\}$ with $\hat{x} \in U$, 
where $f_{K_1} = |\mu_{K_1}|^2$ and $f_{K_2} = |\mu_{K_2}|^2$.
\end{proposition}
\begin{proof} Let $\mathfrak{k}_1 = \stab(y)$ and 
$\mathfrak{k}_2 = \mathfrak{k}_1^\perp$. Since $\fk$ is abelian, both
$\fk_1$ and $\fk_2$ are Lie subalgebras, and $\fk = \fk_1 \oplus \fk_2$.
Then $\mu_K = \mu_{K_1} \oplus \mu_{K_2}$, and 
$|\mu_K|^2 = |\mu_{K_1}|^2 + |\mu_{K_2}|^2$, so the result follows immediately
from Lemma \ref{KeyEstimate1} and inequality (\ref{KeyInequality3}).
\end{proof}

\begin{proof}[Proof of Theorem \ref{GlobalEstimate}]
By Proposition \ref{ReduceToTorus} we will assume that $K$ is a torus.
By Proposition \ref{StrongerCondition}, it suffices to show that for
each $\mu_c \in \fk$, there is some $\epsilon > 0$ so that
inequality (\ref{GlobalInequality}) holds when
$|\mu(x) - \mu_c| < \epsilon$.
Additionally, since $0$ is always a critical point of $f$, it suffices
to prove the estimate only on $V \setminus \{0\}$.
Furthermore, it suffices 
to show that each point $\hat{y}$ of the unit sphere has a neighbourhood $U$
such that the estimate holds for all $x$ with $\hat{x} \in U$,
since by compactness we can choose finitely many
such neighbourhoods to cover the unit sphere, and this yields the
inequality on $V \setminus \{0\}$.

We will prove the estimate by induction on the dimension of $K$. First
suppose $\dim K = 1$. Then we may assume without loss of generality that 
$K$ acts locally freely on $V \setminus \{0\}$, since otherwise
the fundamental vector field $v_\xi(x)$ vanishes on a nontrivial subspace
and we can restrict our attention to its orthogonal complement.
Then Proposition \ref{FreeCase} yields the desired neighbourhoods and estimates.

Now assume that $\dim K > n$ and we have proved the estimate for tori
of dimension $\leq n$. Without loss of generality, we can assume that there is 
no nonzero vector in $V$ which is fixed by all of $K$ (since we can restrict to
its orthogonal complement).
Let $\hat{y}$ be some point in the unit sphere, and let
$\fk_1 = \stab(\hat{y})$. If $\fk_1 = 0$, we may apply
Proposition \ref{FreeCase} to get a neighbourhood $U_{\hat{y}}$ and a constant 
$k_{\hat{y}}$ such that the estimate holds on $U_{\hat{y}}$.
Otherwise,  let $\fk_2 = \fk_1^\perp$ so that 
$\fk = \fk_1 \oplus \fk_2$,  corresponding to subtori $K_1$ and $K_2$.
Then we may apply Proposition
\ref{InductionStep} to find a neighbourhood $U_{\hat{y}}$ so that
\[ |\nabla f(x)| 
\geq k \left( |\nabla f_{K_1}(x)| + |\nabla f_{K_2}(x)| \right), \]
holds for all $x$ with $\hat{x} \in U_{\hat{y}}$.
Let $P_i: \fk \to \fk_i$, $i = 1, 2$ be the orthogonal projections, 
$\mu_{c,i} = P_i \mu_c$,
and $f_{c,i} = |\mu_{c,i}|^2$. Since $\mu_{K_i}$ are the moment maps for the
action of $K_i$, which are tori of dimension $\leq n$, we may apply
the induction hypothesis to find
a neighbourhood $U$ of $\hat{y}$ and constants $\epsilon > 0$, $c > 0$
so that
\begin{eqnarray*}
|\nabla f_{K_1}(x)| &\geq& k' |f_{K_1}(x) - f_{c,1}|^\frac{3}{4}, \\
|\nabla f_{K_2}(x)| &\geq& k' |f_{K_2}(x) - f_{c,2}|^\frac{3}{4},
\end{eqnarray*}
for all $x$ such that $\hat{x} \in U$ and $|\mu_{K_i}(x) - \mu_{c,i}| < \epsilon$.
For any non-negative numbers $a,b$, we have
$a^\frac{3}{4} + b^\frac{3}{4} \geq (a + b)^\frac{3}{4}$,
so we obtain
\[ |\nabla f(x)|
\geq k''\left(|f_{K_1}(x) - f_{c,1}| + |f_{K_2}(x) - f_{c,2}|\right)^\frac{3}{4}
\geq k''|f(x) - f_c|^\frac{3}{4}, \]  
whenever $|\mu(x) - \mu_c| < \epsilon$, as desired.

To obtain the estimate for the functions $|\muc|^2$ and $|\muhk|^2$, 
we simply note that by Lemma \ref{CrossTermVanish}, the norm of the gradient 
is bounded below by a sum of terms of the form $\left| \nabla |\mu_i|^2 \right|$,
and since we can bound each term individually we obtain a bound for the sum.
\end{proof}

\begin{remark} \label{SecondNeemanRemark}
As pointed out in Remark \ref{FirstNeemanRemark}, Theorem \ref{GlobalEstimate}
is a generalization of \cite[Theorem A.1]{NeemanQuotientVarieties}
and much of the argument is similar.
The main new ingredients are Lemma \ref{KeyEstimate2} and
Proposition \ref{FreeCase}, which are absolutely
essential in handling the general case of a nonzero constant term
in the moment map.
\end{remark}
%---------------------------------------------------------------------
      
%---------------------------------------------------------------------
\section{Toric Hyperk\"ahler Orbifolds}
\label{HypertoricSection}
%---------------------------------------------------------------------
\subsection{Notation and Definitions} 
Let $T$ be a subtorus of the standard $N$-torus $(S^1)^N$,
with quotient $K := (S^1)^N / T$. We have a short exact sequence
\begin{equation} \label{ToriSES}
1  \to  T  \xrightarrow{i}  (S^1)^N  \xrightarrow{\pi}  K  \to  1.
\end{equation}
Taking Lie algebras, we have
\begin{equation}\label{SES}
0  \to  \ft \xrightarrow{i}  \RR^N  \xrightarrow{\pi}  \fk  \to 0,
\end{equation}
\begin{equation}
\label{DualSES}
0 \to \fkd \xrightarrow{\pi^\ast} \RR^N \xrightarrow{i^\ast} \ftd \to 0.
\end{equation}
Recall the standard Hamiltonian action of $(S^1)^N$ on $\CC^N$.
This restricts to a Hamiltonian action of $T$ on $\CC^N$, and hence induces 
an action on  $T^\ast \CC^N$. 
For a generic $(\alpha, \beta) \in \ftd \oplus \ftd_\CC$
we define $M(\alpha, \beta)$ to be the quotient
\begin{equation}
M(\alpha, \beta) := T^\ast \CC^N \rreda{(\alpha, \beta)} T,
\end{equation}
which is a toric hyperk\"ahler orbifold \cite{BielawskiDancerHypertoric}. 
There is a residual Hamiltonian action of $K$ on $M$.
The homeomorphism type of $M(\alpha, \beta)$ is independent of 
$(\alpha, \beta)$ as long as $(\alpha, \beta)$ is generic, so we will
often write $M$ instead of $M(\alpha, \beta)$. 

We can organize the data determining $M$ as follows. We will assume
for the moment that
$M$ is taken to be the reduction at $(\alpha, 0)$ with $\alpha$ generic.
Let $\{e_j\}$ be the standard basis of $\RR^N$. Then we obtain a collection
$A := \{u_j\}$ of weights defined by $u_j := i^\ast(e_j) \in \ftd$, as well as 
a collection of normals $\{n_i\}$ defined by $n_i = \pi(e_i) \in \fk$.
Note that we allow repetitions, i.e. $u_i$ and $u_j$ are considered to be
distinct elements of $A$ for $i \neq j$ even if $u_i = u_j$ as elements of 
$\ftd$, and similarly for the normals.
Using the inner product on $\ft$ induced by the embedding $\ft \into \RR^N$,
we can identify $\ft \iso \ftd$ and we will think of the weights $u_j$
as elements of $\ft$ rather than $\ftd$ whenever it is convenient to do so.
Pick some $d \in \RR^N$ such that $i^\ast(d) = \alpha$.
Then we can define affine hyperplanes $H_i$ by
\begin{equation}
H_i = \{ x \in \fkd \suchthat \inner{n_i}{x} - d_i = 0 \},
\end{equation}
as well as half-spaces
\begin{equation}
H_i^\pm = \{ x \in \fkd \suchthat \pm(\inner{n_i}{x} - d_i) \geq 0 \}.
\end{equation}
This arrangement of hyperplanes will be denoted by $\mathcal{A}$.
It is shown in \cite{BielawskiDancerHypertoric} that the arrangement
$\mathcal{A}$ plays a role in toric hyperk\"ahler geometry analogous
to that of the moment polytope in symplectic toric geometry. In
particular, the arrangement $\mathcal{A}$ determines $M$ up to
equivariant hyperk\"ahler isometry.

\begin{definition} \label{DefnJ}
Let $J \subseteq \{1, \cdots, N\}$, and define a subspace
\begin{equation} \label{DefnTJ}
\ft_J := \Span \{u_j \suchthat j \in J\} \subset \ft
\end{equation}
with corresponding subtorus $T_J \subset T$. We will call the set $J$
\emph{critical} if the following condition is satisfied: $u_j \in \ft_J$
if and only if $j \in J$. 
\end{definition}

If $J$ is critical, we define a subspace $V_J \subset \CC^N$ by
\begin{equation} \label{DefnVJ}
V_J := \Span\{e_j \suchthat j \in J\}.
\end{equation}
The action of $T_J$ preserves $V_J$,
and we may take the hyperk\"ahler quotient 
\begin{equation} \label{DefnMJ}
M_J := T^\ast V_J \rred T_J.
\end{equation}
We will always assume that the reduction is taken at a generic
regular value.

\begin{remark} The critical sets $J$ are precisely the \emph{flats} of
the matroid associated to the collection of vectors $\{u_j\}$.
However, we do not wish to assume familiarity with matroids, so we
choose to avoid using this language. See \cite{HauselSturmfels} for a detailed
discussion of the relation between the geometry of toric hyperk\"ahler 
varieties and the combinatorics of matroids, and \cite{OrientedMatroids}
for matroids in general. 
\end{remark}

The inclusion $i:T \into (S^1)^N$ induces a surjective map of rings
$i^\ast: H_{(S^1)^N}^\ast \to H_T^\ast$, so that 
$H_T^\ast \iso \QQ[u_1, \ldots, u_N] / \ker i^\ast$ as a graded ring.
By abuse of notation we will write $u_j$ to denote
its image in $H_T^\ast$. (Note that we also use the symbol $u_j$ to
denote the vectors $i^\ast(e_j)$, but no confusion should arise as it
should be clear from context which of the two meanings is intended.)

If $J$ is a subset of $\{1, \ldots, N\}$, then we define a class 
$u_J \in H_T^\ast$ by
\begin{equation} \label{EulerClassJ}
u_J := \prod_{i \in J^c} u_i,
\end{equation}
and note that the product is taken over the \emph{complement} of $J$.

%---------------------------------------------------------------------
\subsection{Analysis of the Critical Sets}  
We now consider a quotient of the form $M(\alpha, \beta)$ with
$\beta$ a regular value of $\muc$. Note that this includes quotients
of the form $M(\alpha, 0)$ as a special case, since we can
always rotate the hyperk\"ahler frame. We identify $T^\ast \CC^N$ with 
$\CC^N \times \CC^N$ and use coordinates $(x, y)$. Shifting $\muc$ by $\beta$, 
we can take it to be
\begin{equation} \label{MomentMapFormula} 
\muc(x,y) = \sum_{i=1}^N x_i y_i u_i - \beta,
\end{equation}
and we will consider Morse theory with the function $f = |\muc|^2$.

\begin{proposition} 
\label{CriticalSets} For a generic parameter $\beta$,
the critical set of $f$ is the disjoint union of sets $C_J$,
where the union runs over all critical subsets $J \subseteq \{1, \ldots, N\}$,
and the sets $C_J$ are defined by
\begin{equation}
C_J = \left( \bigcap_{i \in J} \{ u_i \cdot \muc = 0\} \right) 
\cap \left( \bigcap_{j \not\in J} \{(x_j, y_j) = 0\} \right).
\end{equation}
The Morse index of $C_J$ is given by $\lambda_J = 2(N - \# J)$.
Up to a nonzero constant, the $T$-equivariant Euler class of the negative
normal bundle to $C_J$ is given by the restriction of the class
$u_J$ to $H_T^\ast(C_J)$, where $u_J$ is defined by (\ref{EulerClassJ}).
The $T$-equivariant Poincar\'e series of $C_J$ is equal to
$(1-t^2)^{-r} P(M_J)$, where $r$ is the codimension of $T_J$ in $T$ and
$M_J$ is the quotient defined by (\ref{DefnMJ}). 
\end{proposition}
\begin{proof} Using equation (\ref{MomentMapFormula}), we see that
\[ |\nabla f(x,y)|^2 = 
 4 \sum_j \left(|x_j|^2 + |y_j|^2 \right) |u_j \cdot \muc(x,y)|^2. \]
Since this is a sum of non-negative terms, if $\nabla f(x,y) = 0$, each term in 
the sum must be 0. Thus for each $j$, we must have either that $x_j = y_j = 0$ or
that $u_j \cdot \muc = 0$. Let us fix some particular critical point 
$(x_c,y_c) \in T^\ast \CC^N$, and 
let $J$ be the set of indices $j \in \{1, \ldots, N\}$ 
for which $u_j \cdot \muc(x_c,y_c) = 0$.
By construction $J$ is critical and $(x_c, y_c) \in C_J$. 
Hence every critical point is contained in $C_J$ for some critical set $J$.
Conversely, $\nabla f = 0$ on $C_J$ by construction, so we see that
$\Crit f = \cup_J C_J$, where the union runs over critical sets $J$.
Note that $C_J \neq \emptyset$ when $J$ is critical.

To see that the union is disjoint, write $\muc = \mu_J + \mu_{J^c}$, where
\begin{eqnarray}
\label{DefnMuJ}
\mu_J(x,y) &=& \sum_{i \in J} x_i y_i u_i - \beta_J, \\
\label{DefnMuJc}
\mu_{J^c}(x,y) &=& \sum_{i \not\in J} x_i y_i u_i - \beta_J^\perp,
\end{eqnarray}
$\beta_J$ is the projection of $\beta$ to $\ft_J$, and 
$\beta_J^\perp = \beta - \beta_J$. Then at $(x_c, y_c)$, we have
$\mu_J(x_c, y_c) = 0$ and $\mu_{J^c}(x_c, y_c) = -\beta_J^\perp$.
Thus on $C_J$, $\muc$ takes the value $-\beta_J^\perp$.
For generic $\beta$, we have $\beta_J^\perp \neq \beta_{J'}^\perp$
for $J \neq J'$, hence $C_J \cap C_{J'} = \emptyset$ for $J \neq J'$.

To determine the Morse index of $C_J$, we compute
\[ |\muc(x,y)|^2 = |\mu_J(x,y)|^2 + |\mu_{J^c}(x,y)|^2 
 + 2 \mathrm{Re} \inner{\mu_J(x,y)}{\mu_{J^c}(x,y)}. \]
The term $|\mu_J(x,y)|^2$ has an absolute minimum at $(x_c, y_c)$,
and so does not contribute to the Morse index. Since $\beta_J^\perp$ is
orthogonal to $\mu_J(x,y)$ for all $(x,y)$, the third term can be
rewritten as
\[ 2 \mathrm{Re} \inner{\mu_J(x,y)}{\mu_{J^c} + \beta_J^\perp}. \]
Looking at the expressions
(\ref{DefnMuJ}) and (\ref{DefnMuJc}) for $\mu_J$ and $\mu_{J^c}$, we see that
at $(x_c, y_c)$, $\mu_J$ vanishes to first order, whereas 
$\mu_{J^c} + \beta_J^\perp$
vanishes to second order. Hence the inner product of these terms vanishes to 
third order and does not affect the Morse index. Thus the Morse index is
determined solely by the second term, which is
\begin{equation} \label{RelevantTerm}
|\mu_{J^c}(x,y)|^2 
 = |\beta_J^\perp|^2 - 2 \mathrm{Re} 
 \sum_{i \in J^c} \inner{\beta_J^\perp}{u_i} x_i y_i + \textrm{fourth order}.
\end{equation}
For generic $\beta$, we have $\inner{\beta_J^\perp}{u_i} \neq 0$ for all
$i \in J^c$, and since each term in the sum is the real part of the 
holomorphic function $\inner{\beta_J^\perp}{u_i} x_i y_i$ it must contribute
$2$ to the Morse index. Hence the Morse index is 
$\lambda_J := 2\# J^c = 2(N - \#J)$. Since the $j$th factor of $(S^1)^N$
acts on $(x_j, y_j)$ with weight $(1, -1)$, this also shows that the
equivariant Euler class is given by a nonzero multiple of $u_J$
(defined by (\ref{EulerClassJ})), as claimed.

Finally, we compute the equivariant Poincar\'e series of $C_J$.
Let $V_J \in \CC^N$ be defined as above, and let $K_J \subset T$ be the
subtorus that acts trivially on $V_J$. Then we have an isomorphism
$T \iso T_J \times K_J$. Let $r$ be the dimension of $K_J$, which is the
codimension of $T_J$ in $T$. The moment map for the action of $T_J$ on
$T^\ast V_J$ is given by the restriction of $\mu_J$ (as defined by
equation (\ref{DefnMuJ})) to $T^\ast V_J$.
Hence $C_J = \mu_J^{-1}(0) \cap T^\ast V_J$, and
\[ P_T(C_J) = P_{T_J \times K_J}(C_J) 
= (1-t^2)^{-r} P_{T_J}(C_J) = (1-t^2)^{-r} P(M_J). \]
\end{proof} 

\begin{remark} Note that the critical sets $C_J$ are all nonempty, and
that the Morse indices do not depend on $\beta$ (as long as it is generic).
This is due to the fact that $\muc$ is holomorphic. In the real case,
i.e. $|\mur - \alpha|^2$, the critical sets and Morse indices have a much
more sensitive dependence on the level $\alpha$.
\end{remark}

By Theorem \ref{SurjectivityCriterion} and Corollary \ref{SurjectivityForTori},
the hyperk\"ahler Kirwan map 
$\kappa: H_T^\ast \to H^\ast(M)$
is surjective, and its kernel is the ideal generated by the equivariant
Euler classes of the negative normal bundles to the components of the critical
set. Since we described these explicitly in Proposition \ref{CriticalSets},
we immediately obtain the following description of $H^\ast(M)$.
\begin{theorem} \label{HypertoricCohomologyRing}
The cohomology ring $H^\ast(M)$ is isomorphic to $H_T^\ast / \ker \kappa$,
where $\ker \kappa$ is the ideal generated by the classes $u_J$,
for every proper critical set $J \subset \{1, \cdots, N\}$.
\end{theorem} \qed

\begin{remark} The cohomology ring $H^\ast(M)$ was first computed by Konno
\cite[Theorem 3.1]{KonnoHypertoric}. The relations defining $\ker \kappa$
obtained by Konno are not identical to those in Theorem 
\ref{HypertoricCohomologyRing}, 
but it is not difficult to see that they are equivalent.
It was pointed out to us by Proudfoot that this equivalence is a 
special case of Gale duality \cite{OrientedMatroids}.  
\end{remark}
 
\begin{remark} Under the assumption that $M$ is smooth (and not just an
orbifold), the same result holds with $\ZZ$ coefficients \cite{KonnoHypertoric}.
In certain cases, the Kirwan method can be extended to handle 
cohomology with $\ZZ$ coefficients,
provided that the group action satisfies certain additional hypotheses
\cite{TolmanWeitsman}.
\end{remark}

Quotients of the form $M(\alpha, 0)$ inherit an additional $S^1$-action
induced by the $S^1$ action on $T^\ast \CC^N$ given by $t \cdot(x, y) = (x, ty)$.
This action preserves the K\"ahler structure and rotates the holomorphic 
symplectic form (i.e. $t^\ast \omegac = t \omegac$). 
Let us fix some particular $\alpha$ and denote $M := M(\alpha, 0)$.
We would like to understand the $S^1$-equivariant 
cohomology $H_{S^1}^\ast(M)$ (which, unlike the ordinary cohomology of 
$M$, \emph{does} depend on the choice of $\alpha$). 
To compute the $S^1$-equivariant cohomology, it is more convenient 
to work directly with $|\muhk|^2 = |\mur|^2 + |\muc|^2$, where
\begin{eqnarray}
\label{MuRFormula} \mur &=& \sum_i\left(|x_i|^2 - |y_i|^2 \right)u_i - \alpha, \\
\label{MuCFormula} \muc &=& \sum_i x_i y_i u_i.
\end{eqnarray}
By Theorems \ref{SurjectivityCriterion} and \ref{GlobalEstimate} it is 
minimally degenerate and flow-closed, and since it is also $S^1$-invariant 
we can consider
the $T \times S^1$-equivariant Thom-Gysin sequence. The usual arguments
of the Kirwan method extend to the $S^1$-equivariant setting, so
we obtain surjectivity of map $\kappa_{S^1}: H_{T \times S^1}^\ast \to H_{S^1}^\ast(M)$,
and its kernel is generated by the $T \times S^1$-equivariant Euler classes
of the negative normal bundles to the critical sets of $|\muhk|^2$.

To find the critical sets of $|\muhk|^2$ and to compute the equivariant
Euler classes, we can repeat the arguments of Proposition \ref{CriticalSets} 
almost without modification. The components of the critical set 
are again indexed by critical subsets $J$. The only important difference
is that since we now work with 
$T \times S^1$-equivariant cohomology, we have to be more careful 
in computing the equivariant Euler classes. Let us make the identification
$H_{T \times S^1}^\ast \iso H_T^\ast[u_0]$.
When we expand $|\muhk|^2$ about a critical point as in 
the proof of Proposition \ref{CriticalSets}, the 
relevant term is now (cf. equation (\ref{RelevantTerm}))
\[  - 2 \sum_{i \in J^c} \inner{\alpha_J^\perp}{u_i} 
\left(|x_i|^2 - |y_i|^2 \right), \]
where $\alpha_J^\perp$ is defined in a manner analogous to $\beta_J^\perp$ as 
in the proof of Proposition \ref{CriticalSets}.
We see that the $x_i$ term appears with an overall negative sign
if $\inner{\alpha_J^\perp}{u_i} > 0$, otherwise it is the $y_i$ term
that appears with a negative sign. Since $S^1$ acts on $x$ with weight
$0$ and acts on $y$ with weight $1$, and since $T$ acts on $x$ and $y$ with
oppositely signed weights, we find that the equivariant Euler class
is given (up to an overall constant) by
\begin{equation}
\tilde{u}_J := \prod_{i \in J^+} u_i \prod_{j \in J^-} (u_0 - u_j),
\end{equation}
where 
\begin{equation}
J^\pm := \{ i \in J^c \suchthat \pm \inner{\alpha_J^\perp}{u_i} > 0 \}.
\end{equation}
Thus we obtain the following.

\begin{theorem} The $S^1$-equivariant cohomology $H_{S^1}^\ast(M)$ is
isomorphic to 
\[ H_{T \times S^1}^\ast / \ker \kappa_{S^1}, \]
where $\ker \kappa_{S^1}$ is the ideal generated by the classes $\tilde{u}_J$,
for every proper critical subset $J$. 
\end{theorem} \qed
 
\begin{remark} The $S^1$-equivariant cohomology rings were first
computed by Harada and Proudfoot 
\cite{HaradaProudfoot}. As in Theorem \ref{HypertoricCohomologyRing}, 
our description of the kernel is dual (but equivalent).
In this instance, the duality is for \emph{oriented} matroids.
\end{remark}

\begin{remark} Note that unlike the ordinary cohomology ring,
the $S^1$-equivariant cohomology ring depends explicitly on the 
parameter $\alpha$.
\end{remark}

%---------------------------------------------------------------------
\subsection{Hyperk\"ahler Modifications}
There is a natural operation called \emph{modification} 
\cite{DancerSwannModifying} defined on 
hyperk\"ahler manifolds (or orbifolds) equipped with a Hamiltonian
$S^1$ action, which is the hyperk\"ahler
analogue of symplectic cutting \cite{LermanCuts}.
We will show in this section that the critical sets $C_J$ can be
understood inductively in terms of modifications and quotients.

Let $M := T^\ast \CC^N \rred T$ be a toric hyperk\"ahler orbifold.
This has a residual Hamiltonian action of a torus $K$. Fix some
particular $S^1$ subgroup of $K$.
Then we can consider $M \times T^\ast \CC$, which has two commuting
$S^1$ actions, diagonal and anti-diagonal. We define the modification
$\tilde{M}$ of $M$ with respect to this $S^1$ action to be the hyperk\"ahler 
quotient 
\begin{equation} \label{DefnMTilde}
 \tilde{M} := (M \times T^\ast \CC) \rred S^1 
= T^\ast \CC^{N+1} \rred (T \times S^1),
\end{equation}
where the quotient is by the anti-diagonal $S^1$.
We also consider the quotient
\begin{equation} \label{DefnMHat}
\hat{M} := M \rred S^1 = T^\ast \CC^N \rred (T \times S^1).
\end{equation}
We will use the notation $\tilde{T} := T \times S^1$ so 
that $\tilde{M}$ and $\hat{M}$ are quotients by $\tilde{T}$.
Let $\mu, \tilde{\mu}$, and
$\hat{\mu}$ denote the respective (complex) moment maps, and let
$A$, $\tilde{A}$, and $\hat{A}$ denote the respective collections of
weights. The critical
sets of $|\mu|^2$, $|\tilde{\mu}|^2$, and $|\hat{\mu}|^2$ are 
defined with respect to  $A, \tilde{A}$, and $\hat{A}$.
We can relate the weights $\tilde{A}$ and $\hat{A}$ corresponding
to a modification and quotient of $M$ to the weights $A$ as follows.

\begin{lemma} \label{Modification}
Let $\tilde{A} = \{\tilde{u}_j\}_{j=1}^{N+1}$. 
Then $\hat{A} = \{\tilde{u}_j\}_{j=1}^{N}$ and $A = \{u_j\}_{j=1}^N$, where
$u_j$ is the image of $\tilde{u}_j$ after quotienting by
$\Span\{\tilde{u}_{N+1}\}$.
\end{lemma}
\begin{proof} The weights are determined by the embeddings
$\ft \into \RR^N$, $\tilde{\ft} \into \RR^{N+1}$, and $\tilde{\ft} \into \RR^N$.
Note that $\tilde{\ft} \iso \ft \oplus \RR$.
If we pick a basis of $\ft$ (and use the standard basis of $\RR^N$),
we can represent the embedding $\ft \into \RR^N$ by some matrix $B$.
The $S^1$ action on $M$ is determined by
specifying its weights on $\CC^N$; this is equivalent to adjoining a column
to $B$. This gives the matrix $\hat{B}$ determining $\tilde{\ft} \into \RR^N$.
Finally, to obtain the modification $\tilde{M}$, we let the $S^1$ act
on an additional copy of $\CC$ with weight $-1$. This amounts to
adjoining the row $(0, \ldots, 0, -1)$ to $\hat{B}$, to obtain
$\tilde{B}$ determining $\tilde{\ft} \into \RR^{N+1}$.
Since the weights $u_j$, $\tilde{u}_j$, and $\hat{u}_j$ correspond to
the rows of $B$, $\tilde{B}$, and $\hat{B}$, respectively,
the result follows from this description.
\end{proof}

Now that we understand the relationship between the weights, we describe
the relationship between the critical subsets $J$, which are defined
with respect to the weights. We will say that 
$J \subseteq \{1, \ldots, N+1\}$ is critical for $A$ (respectively
$\tilde{A}, \hat{A}$) if it indexes a component of the critical set
of $|\mu|^2$ (respectively $|\tilde{\mu}|^2$, $|\hat{\mu}|^2$).
If $N+1 \in J$ then we will not consider it to be critical for $A$ 
(cf. Definition \ref{DefnJ}).

\begin{lemma} \label{CriticalRelation}
 Let $J \subseteq \{1, \cdots, N\}$. Then
\begin{enumerate}
\item $J$ is critical for $A$ if and only if $J \cup\{N+1\}$
is critical for $\tilde{A}$.
\item If $J$ is critical for $A$, then $J$ is critical
for $\hat{A}$.
\item $J$ is critical for $\hat{A}$ if and only if
at least one of $J$ or $J \cup \{N+1\}$ is critical for $\tilde{A}$.
\end{enumerate}
\end{lemma}
\begin{proof} 
Let $\tilde{A} = \{\tilde{u}_j\}_{j=1}^{N+1}$. 
By the previous lemma, 
$\hat{A} = \{\tilde{u}_j\}_{j=1}^N$ and $A = \{u_j\}_{j=1}^N$,
where $u_j$ is the image of $\tilde{u}_j$ after quotienting by $\Span\{u_{N+1}\}$.

To prove (1), first suppose that $J$ is critical for $A$.
Suppose that there is some $\tilde{u}_i \in \tilde{\ft}_{J\cup\{N+1\}}$.
If $i = N+1$ then $i \in J \cup\{N+1\}$ and there is nothing to check,
so suppose that $i \neq N+1$. Applying the quotient map, we see that
$u_i \in \ft_J$ (since $\tilde{u}_{N+1}$ goes to 0), and since $J$ was
critical for $A$ we see that $i \in J \subset J \cup \{N+1\}$. Hence
$J \cup \{N+1\}$ is critical for $\tilde{A}$. 

Conversely, suppose that $J \cup \{N+1\}$ is critical for $\tilde{A}$.
Suppose that there is some $u_i \in \ft_J$. Then if
$u_i = \sum_{j \in J} a_j u_j$, we see that
$\tilde{u}_i - \sum_{j \in J} a_j \tilde{u}_j$ is in the kernel of the projection,
and so is some multiple of $\tilde{u}_{N+1}$. Hence 
$\tilde{u}_i \in \tilde{\ft}_{J \cup \{N+1\}}$. Since $J \cup \{N+1\}$ is
critical for $\tilde{A}$, we must have $i \in J \cup \{N+1\}$.
But $i \neq N+1$ by assumption, so we have $i \in J$.

To prove (2), suppose that $J$ is critical for $A$.
If $\tilde{u}_i \in \tilde{\ft}_J$, then applying the quotient map we find
$u_i \in \ft_J$. Hence $i \in J$.

To prove (3), first suppose that $J \cup \{N+1\}$ is critical for
$\tilde{A}$. Then by (1), $J$ is critical for $A$, and by (2) $J$
is critical for $\hat{A}$. On the other hand, if $J$ is critical
for $\tilde{A}$, then it is certainly critical for $\hat{A}$.
This establishes one direction.

Conversely, suppose that $J$ is critical for $\hat{A}$.
If $\tilde{u}_{N+1} \not\in \tilde{\ft}_J$ then $J$ is critical for
$\tilde{A}$; otherwise $\tilde{u}_{N+1} \in \tilde{\ft}_{J}$ and thus
$J \cup \{N+1\}$ is critical for $\tilde{A}$.
\end{proof}  

We rephrase the preceding lemma as the following trichotomy:
\begin{lemma} \label{Trichotomy}
Let $J \subseteq \{1, \cdots, N\}$ be a 
critical subset with respect to $\hat{A}$. Then
exactly one of the following cases occurs:
\begin{enumerate}
\item $J$ is critical for $\tilde{A}$ and $J$ is not
critical for $A$.
\item $J$ is critical for $A$ and both $J$ and $J \cup \{N+1\}$
are critical for $\tilde{A}$.
\item $J$ is critical for $A$ and $J \cup\{N+1\}$ is 
critical for $\tilde{A}$, while $J$ is not.
\end{enumerate}
Moreover, every critical subset for $A$ and $\tilde{A}$
occurs as exactly one of the above.
\end{lemma} \qed

\begin{theorem} If $\tilde{M}$ is a hyperk\"ahler modification of
$M$ and $\hat{M}$ is the corresponding quotient, then
\begin{equation} \label{PoincareRelation}
P(\tilde{M}) = P(M) + t^2 P(\hat{M}).
\end{equation}
\end{theorem}
\begin{proof} We will prove this by induction on
 $N = \# A$, the number of the number of weights (equivalently, the number of
hyperplanes in $\mathcal{A}$). The base case can be verified
easily, so we assume that the result is true for modifications
$(\tilde{M}', M', \hat{M}')$, where $M'$ is a quotient of $T^\ast \CC^{N'}$,
with $N' < N$.

By Theorems \ref{SurjectivityCriterion} and
\ref{GlobalEstimate}, the functions $f = |\muc|^2$, 
$\tilde{f} = |\tilde{\mu}_\CC|^2$, and $\hat{f} = |\hat{\mu}_\CC|^2$ 
are equivariantly perfect.
$M$ is a quotient of $T^\ast \CC^N$ by a torus of rank $d$,
while $\tilde{M}$ and $\hat{M}$ are quotients of $T^\ast \CC^{N+1}$
and $T^\ast \CC^N$, respectively, by a torus of rank $d+1$.
Hence we have
\begin{eqnarray*}   
\frac{1}{(1-t^2)^d} 
&=& \sum_C t^{\lambda_C} P_T(C), \\
\frac{1}{(1-t^2)^{d+1}} 
&=& \sum_{\tilde{C}} t^{\tilde{\lambda}_C} P_{\tilde{T}}(\tilde{C}), \\
\frac{1}{(1-t^2)^{d+1}} 
&=& \sum_{\hat{C}} t^{\hat{\lambda}_C} P_{\hat{T}}(\hat{C}).
\end{eqnarray*}
Since $0$ is the absolute minimum in each case, we obtain
\[ \frac{1}{(1-t^2)^d} = P(M) + \sum_{C, f(C) > 0} t^{\lambda_C} P_T(C), \]
and similarly for $\tilde{M}$ and $\hat{M}$. Since 
\[ \frac{1}{(1-t^2)^{d+1}} = \frac{1}{(1-t^2)^d} + \frac{t^2}{(1-t^2)^{d+1}} = 0, \]
we just have to show that
\[ \sum_{\tilde{C}, \tilde{f}(\tilde{C}) > 0} t^{\tilde{\lambda}_C} P_{\tilde{T}}(\tilde{C})
= \sum_{C, f(C) > 0} t^{\lambda_C} P_T(C)
 +  \sum_{\hat{C}, \hat{f}(\hat{C}) > 0} t^{\hat{\lambda}_C+2} P_{\hat{T}}(\hat{C}) \]
to obtain the desired recurrence relation among the Poincar\'e polynomials
of $M$, $\tilde{M}$, and $\hat{M}$. By Lemma \ref{Trichotomy}, there
is a trichotomy relating the critical sets of $\tilde{f}$ to the critical
sets of $f$ and $\hat{f}$. We will consider each case separately.

Case (1): We have $\tilde{C} = \tilde{C}_J$ where $J$ is critical with
respect to $\tilde{A}$ and $\hat{A}$ but not with respect to $A$. From
Proposition \ref{CriticalSets}, we have 
$\tilde{\lambda}_J = \hat{\lambda}_J + 2$ and 
$\tilde{C}_J = \hat{C}_J \times (0, 0)$. Hence 
$t^{\tilde{\lambda}_J} P_{\tilde{T}}(\tilde{C}_J) 
= t^{\hat{\lambda}_J+2} P_{\hat{T}}(\hat{C}_J)$.

Case (2): We have $\tilde{C} = \tilde{C}_J$ where $J$ is critical with
respect to $\hat{A}$ and $J \cup \{N+1\}$ is also critical for
$\tilde{A}$. Then $\tilde{\lambda}_J = \hat{\lambda}_J + 2$
and $\tilde{\lambda}_{J\cup\{N+1\}} = \lambda_J$. The terms
involving $\tilde{C}_J$ and $\hat{C}_J$ are equal as in case (1).
Since both $J$ and $J \cup \{N+1\}$ are critical for $\tilde{A}$, it must
be that $\tilde{u}_{N+1} \not\in \tilde{\ft}_J$. Hence
$\tilde{T}_{J \cup \{N+1\}} \iso \tilde{T}_J \times S^1$, where the 
last factor is generated by $\tilde{u}_{N+1}$, and we find that
$\tilde{M}_{J \cup \{N+1\}} \iso M_J$. Hence 
$P_{\tilde{T}}(\tilde{C}_{J \cup \{N+1\}}) = P_T(C_J)$.

Case (3): $\tilde{C} = \tilde{C}_{J \cup \{N+1\}}$, where
$J$ is critical for $A$ and $\hat{A}$ but not for $\tilde{A}$.
By Proposition \ref{CriticalSets}, we have
$P_{\tilde{T}}(\tilde{C}_{J \cup\{N+1\}}) = (1-t^2)^{-r} P(\tilde{M}_{J \cup \{N+1\}})$,
$P_T(C_J) = (1-t^2)^{-r} P(M_J)$, and 
$P_{\hat{T}}(\hat{C}_J) = (1-t^2)^{-r} P(\hat{M}_J)$, where $r$ is
the codimension of $T_J$ in $T$. Thus we have to show that
\begin{equation} \label{InductionRelation}
P(\tilde{M}_{J \cup \{N+1\}}) = P(M_J) + t^2 P(\hat{M}_J).
\end{equation}
But $\tilde{M}_{J \cup \{N+1\}}$ is a modification of $M_J$, and $\hat{M}_J$
is the corresponding quotient of $M_J$. Since 
$J$ is a proper subset of $\{1, \ldots, N\}$, the relation
(\ref{InductionRelation}) is true by induction.
\end{proof}

The relation (\ref{PoincareRelation}) is equivalent to
the following recurrence relation among the Betti numbers
of $M$, $\tilde{M}$, and $\hat{M}$:
\[ \tilde{b}_{2k} = b_{2k} + \hat{b}_{2k} + \hat{b}_{2k-2}. \]
If we let $d_k$ denote the number of $k$-dimensional facets of
the polyhedral complex generated by the half-spaces $H_i^\pm$ 
(with $\tilde{d}_k$ and $\hat{d}_k$ defined similarly), 
it turns out \cite{BielawskiDancerHypertoric} that these satisfy the 
same relation:
\[ \tilde{d}_k = d_k + \hat{d}_k + \hat{d}_{k-1}. \]
Since any toric hyperk\"ahler orbifold can be constructed out of 
a finite sequence of modifications starting with $T^\ast \CC^n$, an
 easy induction argument then yields the following explicit description
of $P(M)$.
\begin{corollary}[{Bielawski-Dancer \cite[Theorem 7.6]{BielawskiDancerHypertoric}}]
\label{HypertoricPoincare}
The Poincar\'e polynomial of $M$ is given by
\[ P(M) = \sum_k d_k (t^2 - 1)^k. \]
\end{corollary} \qed
%---------------------------------------------------------------------

%---------------------------------------------------------------------
% Bibliography
%---------------------------------------------------------------------
\bibliographystyle{amsplain}

%---------------------------------------------------------------------
 
%---------------------------------------------------------------------
\end{document}